\documentclass[envcountsect]{llncs}
\usepackage[utf8]{inputenc}
\usepackage{amssymb,amsmath,mathrsfs}
\usepackage[english]{babel}
\usepackage[unicode,unicode=true,bookmarks=false]{hyperref}

\newcommand{\RR}{\mathbb{R}}
\newcommand{\NN}{\mathbb{N}}
\newcommand{\ZZ}{\mathbb{Z}}
\newcommand{\QQ}{\mathbb{Q}}
\newcommand{\CC}{\mathbb{C}}

\newcommand{\PrA}{\mathop{\mathbf{PrA}}\nolimits}
\newcommand{\PA}{\mathop{\mathbf{PA}}\nolimits}

\newcommand{\Th}{\mathop{\mathbf{Th}}\nolimits}

\makeatletter

\renewcommand{\epsilon}{\varepsilon}
\renewcommand{\phi}{\varphi}
\newcommand{\sref}[2]{\hyperref[#2]{#1 \ref*{#2}}}
\newcommand{\dref}[2]{\hyperref[#2]{ #1 }}

\newcommand{\Dc}{\mathcal{D}}

\newcommand{\Kc}{\mathcal{K}}
\newcommand{\Lc}{\mathcal{L}}

\newcommand{\eqdef}{\stackrel{\mbox{\tiny\rm def}}{=}}
\newcommand{\ra}{\rightarrow}

\newcommand{\Lra}{\Leftrightarrow}
\newcommand{\lra}{\leftrightarrow}

\spnewtheorem{hyp}{Conjecture}[section]{\bfseries}{\itshape}
\spnewtheorem{ex}{Example}{\bfseries}{\itshape}

\newcommand{\xv}{\overline{x}}
\newcommand{\yv}{\overline{y}}
\newcommand{\av}{\overline{a}}
\newcommand{\bv}{\overline{b}}

\begin{document}
	\author{Alexander Zapryagaev\thanks{The publication was prepared within the framework of the Academic Fund Program at the National Research University Higher School of Economics (HSE) in 2019 (grant No. 19-04-050) and by the Russian Academic Excellence Project "5-100".}}
	\title{Interpretations of Linear Orderings in Presburger Arithmetic}
	\institute{National Research University Higher School of Economics, 6, Usacheva Str., Moscow, 119048, Russian Federation}
	
	\maketitle
	
	\begin{abstract}
		Presburger Arithmetic $\PrA$ is the true theory of natural numbers with addition. We consider linear orderings interpretable in Presburger Arithmetic and establish various necessary and sufficient conditions for interpretability depending on dimension $n$ of interpretation. We note this problem is relevant to the interpretations of Presburger Arithmetic in itself, as well as the characterization of automatic orderings. For $n=2$ we obtain the complete criterion of interpretability.
	\end{abstract}
	
	\section{Introduction}
	
	Presburger Arithmetic, the true theory of natural numbers with addition, is a prime example of a sufficiently well-applied weaker arithmetical theory. As opposed to Peano Arithmetic $\PA$, it is complete, decidable and admits quantifier elimination in an extension of its language.
	
	In this work we consider the interpretations of linear orderings in the standard model $(\NN,+)$ of $\PrA$, that is, Presburger definitions of the ordering relation on natural numbers that, on some Presburger-definable subset of $\NN$ model a given ordering. Presburger arithmetic is weak enough to be unable to encode tuples of numbers by single numbers, so it becomes relevant to study one-dimensional and $n$-dimensiona, $n>1$ interpretations separately. In the work \cite{zp}, co-authored by the current author, it was established that a Hausdorff-type notion of a rank pertaining to an ordering ($VD_*$-rank) gives the simplest necessary condition of interpretability:
	
	\begin{theorem}\label{basordeing}
		All linear orderings $n$-dimensionally interpretable in $(\NN,+)$ have the $VD_*$-rank $n$ or below.
	\end{theorem}

	This immediately implies all interpretable orderings are \textit{scattered}.
	
	However, for $n\ge 2$, even from cardinality arguments, this condition is far from sufficient. This work attempts to tune the classification and obtain a more exact description of interpretable orderings. It turns out that for dimension $2$ all the interpretable orderings can be obtained by limiting the lexicographic ordering on $\NN^k$ for some $k$ onto some $\PrA$-definable set. This is obtained by considering the interpretation of the ordering produced from the initial one by identifying the points at the finite distance. It is conjectured that the same holds for all dimensions $k\ge 1$, thus giving the complete description.
	
	\section{Preliminaries}
	
	\subsection{Presburger Arithmetic}
	
	We consider the language $\Lc^-\eqdef\{=,+\}$.
	
	\begin{definition}
		{\em Presburger Arithmetic} $(\PrA)$ is the elementary theory of the model $(\NN,+)$ of natural numbers with addition.
	\end{definition}
	
	It is easy to define the relations $<$ and $\le$, modulo comparison $\equiv_n$ for all $n\ge 1$ in the model $(\NN,+)$. By fixing some definitions of those as additional axioms, we can extend the signature to $\Lc\eqdef\left\{=,+,<,\{\equiv_n\}_{n\in\NN}\right\}$ and freely switch between $\Lc$-formulas and equivalent $\Lc^{-}$-formulas. A sample definition could be introduced as follows:
	
	\begin{itemize}
		\item $t\neq s\lra\neg(t=s),\:t\le s\lra\exists u(t+u=s)$;		
		\item $t<s\lra t\le s\wedge t\neq s,\:t>s \lra s<t,\:t\ge s\lra s\le t$;		
		\item $\underline{n}t\eqdef \underbrace{t+\ldots+t}_{\text{$n$ times}}$ (formal recursive definition is $((\underline{n-1})t+t)$ with $\underline{0}t=0$ for all $n\in\NN$);
		\item $\underline{n}\lra\underline{n}1$ for all $n\in\NN$;		
		\item $t\equiv_n s\lra\exists u\:(t=\underline{n}u+s\vee s=\underline{n}u+t)$ for all $n\in\NN$.
	\end{itemize}
	
	This extension gives us quantifier elimination in Presburger arithmetic \cite{presburger}:
	
	\begin{theorem}[Presburger, 1929]
		Each formula in the language $\Lc$ is $\PrA$-provably equivalent to a quantifier-free $\Lc$-formula.
	\end{theorem}
	
	Hence, $\PrA$ is complete and decidable, in a sharp contrast with Peano Arithmetic $\PA$ that involves multiplication. Completeness immediately follows from the fact $\PrA$ is the true theory of natural numbers. The second requires manual checking that Presburger decides atomic formulae. 
	
	We note it is ambiguous whether one should choose $\ZZ$ or $\NN$ as the domain of the model. As \cite[p. 2]{haase} points out, theory initially offered by Tarski to Presburger was the integer-based $\Th(\ZZ,=,+,0,1)$, and the method of quantifier elimination was easily extendable to $\Th(\ZZ,=,+,0,1,<)$, which this is equivalent to $\Th(\NN,=,+)$ without order. Indeed, order is definable in naturals as follows:
	
	\begin{center}
		$x<y\equiv\exists z\:y=x+z+1.$
	\end{center}
	
	On the other hand, by defining integers as differences of naturals, we can implement operations on $\ZZ$.
	
	Similarly to $\PA$, $\PrA$ allows non-standard models. Unlike $\PA$, however, where it is impossible to produce an explicit non-standard model by defining some recursive addition and multiplication (Tennenbaum's Theorem), examples of non-standard models of $\PA$ can be given easily.
	
	\begin{example}
		Consider pairs $(a,b)$, where $a\in\QQ_{\ge 0}$, $b\in\ZZ$ except in case when $a=0$, when we define $b\in\NN$. If we define addition on these pairs by component, it can be checked that the structure is a model of $\PrA$. It is non-standard, as its order type is $\NN+\ZZ\cdot\QQ$.
	\end{example}
	
	However, the same construction that is normally used in the studies of $\PA$ allows us to describe the order types of non-standard models of $\PrA$.
	
	\begin{theorem}[On Non-Standard Models]
		Any nonstandard model $\mathfrak{A}\models\PrA$ has the order type $\NN+\ZZ\cdot A$, where $\langle A,<_A\rangle$ is some dense linear order (DLO) without endpoints.
	\end{theorem}
	\begin{proof}
		Each non-standard element of $\mathfrak{A}$ belongs to a fragment fragment isomorphic to $\ZZ$. We call such fragments galaxies: $[a]\eqdef\{a+n\mid n\in\NN\}\cup\{a-n\mid n\in\NN\}$. We introduce the ordering on galaxies as one induced from the original ordering: $[a]<[b]$ means $a<b$ and $[a]\neq[b]$. This ordering is linear, and we may guarantee its denseness and absence of endpoints by giving an example of a point to the left, right, and between any two galaxies $[a]<[b]$. This is done using division $\left\lfloor\frac{x}{n}\right\rfloor:$
		
		\begin{center}
			$\left[\left\lfloor\frac{a}{n}\right\rfloor\right]<[a]<\left[\left\lfloor\frac{a+(\underline{n-1})\cdot b}{n}\right\rfloor\right]<[b]<[\underline{2}b].$
		\end{center}
	\end{proof}
	
	In particular, any countable model of $\PrA$ has the order type of either $\NN$ or $\NN+\ZZ\cdot\QQ.$
	
	There is a simple characterization of Presburger-definable sets, given by combining two classical results.
	
	\begin{definition}
		For vectors $\overline{c},\overline{p_1},\ldots,\overline{p_n}\in\ZZ^m$ we call the set $\{\overline{c}+\sum k_i\overline{p_i}\mid k_i\in\NN\}$ a \textbf{lattice} (or a linear set) generated by $\{p_i\}$. If $\{p_i\}$ are linearly independent, we call the set a \textbf{fundamental lattice}.
	\end{definition}
	
	According to \cite{ginsburg}, Presburger-definable subsets of $\NN^m$ are exactly the unions of a finite number of (possibly intersecting, possibly non-fundamental) lattices (also called \emph{semilinear sets} in literature). However, Ito in \cite{ir}has shown that any set in $\NN^m$ which is a union of a finite number of (possibly intersecting, possibly non-fundamental) lattices (semilinear sets) can be expressed as a union of a finite number of disjoint fundamental lattices. Hence, $\PrA$-definability is equivalent to being a finite disjoint union of fundamental lattices (not necessarily of dimension $m$).
	
	It is now obvious to take the maximal dimension of the sets appearing in such a decomposition as a Presburger-transformation invariant of a infinite definable set; the possibility is given by the following statements:
	
	\begin{theorem}
		Every infinite definable subset $A\subseteq\NN^m,\:m\ge 1$ has a definable bijection onto $\NN^n$ for some $n,$ where $1\le n\le m.$ Furthermore, there is no definable bijection between $\NN^k$ and $\NN^l$, whenever $k\neq l$. Hence, this $n$ is unique.
	\end{theorem}
	
	This was proven by \cite{cluckers} and restated in \cite{zp}. This claim allows one to define \emph{Presburger dimension} $n$ of a infinite set $A\subseteq\NN^m$ according to the existence of a definable bijection between $A$ and $N^n$. It can be shown that this definition is equivalent to the one above.
	
	Additionally, one desires to give a description of Presburger-definable functions. It is as follows:
	
	\begin{definition}
		A function $f\colon A\ra\NN$ where $A$ is definable is called \emph{piecewise polynomial} of degree $\le m$ if there is a decomposition $A=C_1\sqcup\ldots\sqcup C_k$ into fundamental lattices that the restriction of $f$ to each $C_i$ is a rational polynomial of degree $\le m$.
	\end{definition}
	
	All $\PrA$-definable functions are exactly piecewise linear. In this work, the term `piecewise' is used exclusively in this meaning.
	
	\subsection{Interpretations}
	
	We define multi-dimensional first-order non-parametric interpretations, following \cite{tarskimostowski}.
	
	\begin{definition}
		An \textbf{$m$-dimensional interpretation} $\iota$ of some first-order language $\Kc$  in a model $\mathfrak{A}$ consists of first-order formulas of language of $\mathfrak{A}$:
		
		\begin{enumerate}
			\item $D_{\iota}(\overline{y})$ defining the set $\mathbf{D}_{\iota}\subseteq \mathfrak{A}^m$ (domain of interpreted model);			
			\item $P_{\iota}(\overline{x}_1,\ldots,\overline{x}_n)$, for predicate symbols $P(x_1,\ldots,x_n)$ of $\Kc$ including equality;			
			\item $f_\iota(\overline{x}_1,\ldots,\overline{x}_n,\overline{y})$, for functional symbols $f(x_1,\ldots,x_n)$ of $\Kc$.
		\end{enumerate}
	\end{definition}
	
	Here all vectors of variables $\overline{x}$ are of length $m$, and $f_\iota$'s should define graphs of some functions (modulo interpretation of equality).
	
	Naturally, $\iota$ and $\mathfrak{A}$ give a model $\mathfrak{B}$ of the language $\Kc$ on the domain $\mathbf{D}_{\iota}/{\sim_{\iota}}$, where equivalence relation $\sim_{\iota}$ is given by $=_{\iota}(\overline{x}_1,\overline{x}_2)$. We will call $\mathfrak{B}$ the \textbf{internal model}.
	
	If $\mathfrak{B}\models \mathbf{T}$, then $\iota$ is an \emph{interpretation of the theory} $\mathbf{T}$ in $\mathfrak{A}$.
	
	An interpretation is a very natural concept, appearing in mathematics when, for example, Euclidean geometry is interpreted in the theory of real numbers $\RR$ (two-dimensionally, by defining points as pairs of real numbers) in analytic geometry, or the field $\CC$ of complex numbers is two-dimensionally interpreted in $\RR$ by defining $a+bi\lra(a,b)$ and declaring addition and multiplication. We note that in $(\NN,+)$ itself, the field $(\ZZ,+)$ can be interpreted. This is achieved by mapping the negative numbers to odd, positive to even and $0$ to $0$ and defining the addition case-by-case (through non-negative subtraction, which is definable).
	
	We will be interested in interpretations of theories in the standard model of Presburger Arithmetic, that is, $(\mathbb{N},+)$.
	
	\begin{definition}
		An interpretation is called \emph{non-relative} if the domain is trivial, i.e.
		
		\begin{center}
			$\mathfrak{A}\models \forall \overline{y}\; D_\iota(\overline{y}).$
		\end{center}
		
		It is said that an interpretation \emph{has absolute equality} if the symbol $=\in\Kc$ is interpreted in the internal model as the coincidence of two $m$-tuples.
	\end{definition}
	
	We remind that the internal model of a interpretation $\iota$ is a model with the domain $\mathbf{D}_{\iota}/\sim_{\iota}$.
	
	\begin{itemize}
		\item $m_1$-dimensional interpretation $\iota_1$ and $m_2$-dimensional interpretation $\iota_2$ are \emph{isomorphic}, if there is an isomorphism $f$ between the corresponding internal models;
		\item If $f$ can be expressed by a $(m_1+m_2)$-ary formula $F$ we call this isomorphism \emph{definable}.
	\end{itemize}
	
	For the sake of notational simplicity, we will denote the ordering established on the internal model of some linear ordering $L$ interpreted in $(\NN,+)$ as $<_*$, to avoid confusion with the external ordering $<$ in the naturals.
	
	\section{Rank Condition on Linear Orderings}
	
	The motivation to investigate the interpretability of linear orderings in $(\NN,+)$ derives from the following \emph{Visser's conjecture}: $\PrA$ couldn't  be interpreted in any of its finite subtheories (*). We note that in the case of sufficiently rich theories that can formalize syntax, this is a consequence of \emph{reflexivity}, which is ability to prove the consistency of all its finitely axiomatizable subtheories. Hence, the property (*) can be considered a somewhat trace of reflexivity for weaker theories.
	
	To prove (*), it suffices to show that all interpretations of $\PrA$ in $(\mathbb{N},+)$ are definably isomorphic to the trivial one (see \cite[p. 363, Theorem 5.3]{zp}). In order to show this, there are tow steps: show that the isomorphism exists and then show that it is $\PrA$-definable. Proving existence amounts to explaining why the order type of the internal model's order cannot be isomorphic to $\NN+\ZZ\cdot\QQ$.
	
	Let us call an order \emph{scattered} iff it does not have a dense infinite suborder.
	
	Once again, let us call the equivalence class of a point in a linear ordering under the equivalence relation `there is a finite number of points between two given'. We define \emph{condensation} as a function that maps a linear ordering to the induced ordering on its galaxies (fusing the points on finite distance). Now we can introduce a rank of a linear ordering as the minimal number of iterated condensations required to reach a finite ordering. More formally, we introduce the following definition:
	
	\begin{definition}
		Let $(L,<)$ be a linear ordering. Define the family of equivalence relations $\simeq_{\alpha}$ for ordinals $\alpha\in\mathbf{Ord}$ by transfinite recursion ($\lambda$ is a limit ordinal):
		
		\begin{itemize}
			\item $\simeq_0$ is equality;
			\item $a\simeq_{\alpha+1}b\Lra|\{c\in L\mid (a<c<b)\mbox{ or }(b<c<a)\}/{\simeq_{\alpha}}|<\aleph_0$;
			\item $\simeq_{\lambda}=\bigcup\limits_{\beta<\lambda}\simeq_{\alpha}$.
		\end{itemize}
		
		Now we define the $VD_*$-\emph{rank} $\mathrm{rk}(L,<)\in \mathbf{Ord}\cup \{\infty\}$ of the order $(L,<)$ to be the least $\alpha$ such that $L/{\simeq_{\alpha}}$ is finite and $\infty$ otherwise. (By definition $\alpha<\infty$ for all $\alpha\in\mathbf{Ord}$.)
	\end{definition}
	
	Note that the orders $L$ with $\mathrm{rk}<\infty$ are exactly the scattered ones.
	
	In \cite{zp}, the following was shown:
	
	\begin{theorem}[J.~Zoethout, A.~Zapryagaev, F.~Pakhomov, 2016]
		All linear orders $(m\ge 1)$ $m$-dimensionally interpretable in $(\NN,+)$ have $VD_*$-rank $m$ or below.
	\end{theorem}
	\begin{proof}
		We recapitulate the proof here in sketch, for consistency.
		
		Induction on $m$. Case $m=1$ is obvious.
		
		Let there be an $\Lc^{-}$ formula $D(\xv)$ for domain and $\prec_*(\xv,\yv)$ for interpretation of the order relation. Without loss of generality we may assume that $L= \{\av \in \mathbb{N}^m \mid (\mathbb{N},+)\models D(\av)\}$ and $\prec$ is defined by the formula $\prec_*$.
		
		Assume $\mathrm{rk}(L,\prec)>m$. There are infinitely many distinct $\simeq_m$-equivalence classes in $L$, and it is possible to form an infinite chain $\av_0\prec \av_1\prec\ldots $ of elements of $L$ such that $\av_i\not\simeq_m \av_{i+1}$, for each $i$\footnote{Or an infinitely descending chain, if first is impossible.}. We consider intervals $L_i=\{\bv\in L\mid \av_i<\bv<\av_{i+1}\}$. The set $L_i/{\simeq_{m-1}}$ is infinite and $\mathrm{rk}(L_i,\prec)>m-1$.
		
		All $L_i$ are Presburger-definable sets, and that $\dim(L_i)\ge m$, for each $i$.
		
		Consider the parametric family of subsets of $\NN^m$ given by the formula $\yv_1\prec_* \xv\prec_* \yv_2$ and pairs of parameters $\yv_1=\av_i$ and $\yv_2=\av_{i+1}$, for $i\in\mathbb{N}$, which are $L_i$'s. Thus we have infinitely many disjoint sets of the dimension $m$ in the family, which is impossible.
	\end{proof}
	
	For Visser's conjecture, this implies there is no $m$ that the non-scattered order $\NN+\ZZ\cdot\QQ$ can be $m$-dimensionally interpreted in the Standard Model. Thus, the isomorphism exists for all $m\in\NN$.
	
	However, for the problem of the complete description of linear orderings interpretable in $(\NN,+)$ this is obviously but a necessary condition: not all linear orders with $VD_*$-rank $m$ or below are $m$-interpretable (or interpretable at all). Indeed:
	
	\begin{lemma}
		Each scattered linear order of $VD_*$-rank 1 is $1$-dimensionally interpretable in $(\NN,+)$. There are scattered linear orders of $VD_*$-rank 2 that are not interpretable in $(\NN,+)$.
	\end{lemma}
	\begin{proof}
		The interpretability of linear orders with rank $0$ and rank $1$ is obvious.
		
		Since there are uncountably many non-isomorphic scattered linear orders of $VD_*$-rank 2 and only countably many linear orders interpretable in $(\NN,+)$, there is some scattered linear order of $VD_*$-rank 2 that is not interpretable in $(\NN,+)$.
	\end{proof}
	
	We attempt to find the necessary and sufficient conditions for the linear ordering to be $m$-dimensionally interpretable in $(\NN,+)$.
	
	\section{Galaxy Definabilities}
	
	It is clear that a galaxy $[x]$ of a point $x$ in our scattered order $L$ is isomorphic to either of the following: $\ZZ$, $\NN$, $-\NN$, or a finite set.
	
	For each point $\overline{x}$ from the domain $L$ we introduce the following predicates:
		
		\begin{itemize}
			\item $N(x)$, if $[x]\simeq\NN$;			
			\item $\overline{N}(x)$, if $[x]\simeq-\NN$;			
			\item $Z(x)$, if $[x]\simeq\ZZ$;			
			\item $F(x)$ (from \emph{finite}), if $|[x]|<\infty$;			
			\item in particular, $F_n(x)$, if $|[x]|=n$.
		\end{itemize}
	
	\begin{lemma}
		The predicate $T(x,y)$ `$x$ and $y$ are in the same galaxy' is definable.
	\end{lemma}
	\begin{proof}
		For the naturals finiteness is same as boundedness. The formula for `$\{z\mid x<_*z<_*y\}$ is bounded by some $n$ in all dimensions' defines the required.
	\end{proof}
	
	It is a simple exercise to show:
	
	\begin{lemma}
		Let $(L,<)$ be an $(\NN,+)$-interpretable ordering. Then the predicates $Z(x),N(x),\overline{N}(x),F_n(x)$ for all $n\in\NN_{+}$ are $\PrA$-definable.
	\end{lemma}
	
	\section{Condensed Ordering}
	
	Now we may show that the if an order is $m$-dimensionally interpretable, then its condensation is $(m-1)$-interpretable\footnote{Thus, the interpretability of \textit{all condensations} is a necessary condition.}.
	
	\begin{theorem}
		Let $(L,<)$ be an $m$-definable ordering, and $(cL,<')$ its condensation. Then $cL$ is $(m-1)$-definable.
	\end{theorem}
	
	\begin{proof}
		It is not hard to show that $cL$ is $m$-definable, and, furthermore this interpretation can be obtained by limiting the domain of the interpretation of $L$. Indeed, we may fix a point, say, lexicographically minimal, in each galaxy, as shown by the aforementioned definabilities:
		
		\begin{center}
			$$\Dc'(x)=\Dc(x)\wedge\forall y\,\left(T(x,y)\ra\text{$x$ is lexicographically smaller than $y$}\right).$$
		\end{center}
		
		Similarly, it is possible to choose the internal model least point of galaxies isomorphic to $\NN$ and the largest point of those isomorphic to $-\NN$ (we may also define any galaxy isomorphic to $\ZZ$ as two subsequent galaxies isomorphic to $-\NN$ and $\NN$ respectively, splitting them, for example, at the lexicographically minimal point again). This will be more convenient to us.
		
		A less obvious matter is why the dimension actually decreases.
	
		There is an infinite number of galaxies isomorphic to infinite sets (otherwise $cL$ would be finite, and everything is trivial). Additionally, any finite galaxy can only be neighbouring to an infinite one. There are two logical possiblitities: first, we deal with the case when the representative points from infinite galaxies form an $n$-dimensional set. We assume there is an infinite number of galaxies isomorphic to $\NN$.
		
		It suffices to consider the galaxies that are isomorphic to $\NN$. Assume their starting points are forming a wholly $n$-dimensional set; each of galaxies themselves, however, takes at least one dimension. If we take their $k$th points for each particular $k$ (also definable sets), an infinite number of those are $n$-dimensional. The clearest way to see that is to consider the (definable) function of successor and take an $n$-dimensional subset of those first points of the galaxies on which this function is linear. Then the set of second points derived from those will be a linear transformation of the previous one. Continuing this sieving of points to infinity may eliminate all the points; however, it suffices to repeat it only $m^n+1$ times, where $m$ is the least common divisor of all the modulo comparisons in the definition of $L$. We see that there are $m^n+1$ non-intersecting $n$ dimensional sets in $m^n$, which is impossible.
		
		Now we consider the second possibility: the representative points from infinite galaxies are less that $n$-dimensional but the points from finite ones are. This, however,is impossible: consider the definable function mapping each representative point from a finite galaxy to some representative point of an infinite galaxy after it but before the previous finite point, if such infinite point exists; alternatively, this is an inverse of a function that maps some of the infinite points to preceding finite points. Now it is easy to see that this function, when considered in the direction from the (less than $n$-dimensional, ever more so as some of the infinite points are skipped) infinite points to the ($n$-dimensional) finite points (no more than one skipped, as between each two finite points there is an infinite one) is a piecewise linear function, but it increases the dimension.
	\end{proof}
	
	Now we apply the result to give a complete classification of linear orderings interpretable in no more than two dimensions.
	
	We recall the following result, derived from \cite{blakley}, but with a boundary on on the asymptotics of the function established in (\cite[Theorem 6.2]{zp}):
	
	\begin{theorem}\label{bash}
		Let $A$ be a $d\times n$ matrix of integer numbers, function $\varphi_A\colon\mathbb{Z}^d\ra \NN\cup \{\aleph_0\}$ is defined as follows:
		
		\begin{center}
			$\varphi_A(u)\eqdef|\{\overline{\lambda}=(\lambda_1,\ldots,\lambda_n)\in\NN^n\mid A\lambda=u\}|.$
		\end{center}
		
		Then if the values of $\varphi_A$ are always finite, the function $\varphi_A$ is a piecewise polynomial function of a degree $\le n-\mathrm{rk}(A)$.
	\end{theorem}
	
	Consider an ordering $(L,<)$ definable in $\NN^2$. It has rank no more than $2$, and the descriptions of ranks $0$ and $1$ are established.
	
	Assume it has the whole dimension $2$. We already know that its condensation is an ordering fo rank $1$, that is, forms a finite sequence of galaxies. It suffices to consider the case with the condensed ordering $L'=\NN$. As we already know, $L'$ is defined on a one-dimensional set; this means there is a bijection of this defined ordering onto the naturals (the $x$-axis). Furthermore, this mapping can be done in the correct ascending order due to the main result of \cite{zp} that all one-dimensional interpretations of $\PrA$ are isomorphic to trivial. This means:
	
	\begin{lemma}
		The subsets of $\NN$ corresponding to the infinite, finite galaxies and all subtypes thereof (isomorphic to $\NN,-\NN,\ZZ$, having cardinality $k$ for each $k\in\NN$) are definable.
	\end{lemma}
`	
	Indeed, they are definable on $L'$.
	
	Now let us consider $D\subseteq\NN$ corresponding to the finite galaxies. As \sref{Theorem}{bash} shows, the function $f\colon D\ra\NN$, $f\colon k\mapsto|\text{galaxy corresponding to $k$}|$ is piecewise polynomial (in fact, no more than quadratic).
	
	\begin{lemma}
		If an ordering $L$ of rank $2$ is $2$-definable, $L'$ equals to $\NN$, the subsets of $\NN$ corresponding to each type of infinite galaxies are definable, and the function $f\colon k\mapsto|\text{galaxy corresponding to $k$}|$ is definable, then $L$ is definable as a limiting of the lexicographic ordering on $\ZZ^m$ for some $m$ onto some definable set.
	\end{lemma}
	\begin{proof}
		We start with two dimensions, adding more as we need. First, $\{(n,0)\mid n\in\NN\}$ defines $L'$. The cases of infinite galaxies are dealt with easily: if, say, galaxy $0$ is isomorphic to $\NN$,then we need to add $\{(0,n)\mid n\in\NN\}$, similarly for $-\NN$ and $\ZZ$. Finally, consider the finite galaxies. The function of cardinalities is piecewise polynomial; if, say, it is linear then we may just add
		
		\begin{center}
			$\left\{(k,n)\mid n\in[0,|\text{galaxy corresponding to $k$}|]\right\}$;
		\end{center}
		
		for greater cardinalities adding more than one dimension is necessary, but still possible; for example, the function $x\mapsto x^2$ can be represented with the growing squares 
		
		\begin{center}
			$\left\{(k,n_1,n_2)\mid n_1,n_2\in[0,k]\right\}$.
		\end{center}
	\end{proof}
	
	Noting that the case of more complicated orderings of rank $1$ for $L'$ that $\NN$ it is still possible to maintain the description with adding more dimensions, we finally establish:
	
	\begin{theorem}
		Each $2$-definable ordering is definable as a limiting of the lexicographic ordering on $\ZZ^m$ for some $m$ onto some definable set.
	\end{theorem}
	
	The obvious generalization would be to extend the description to all dimensions beyond $2$, and the following work will be dedicated to that.
	
	\section*{Acknowledgements}
	
	The author thanks Fedor Pakhomov for many ideas and discussions and Lev Beklemishev for interest and attention to the project.

\end{document}